\documentclass[11pt,english]{article}
\usepackage{color}
\usepackage[T1]{fontenc}
\usepackage[latin9]{inputenc}
\usepackage[a4paper]{geometry}
\usepackage{textcomp}
\usepackage{amsthm}
\usepackage{amsmath}
\usepackage{amssymb}
\usepackage{esint}
\usepackage{bbm}
\usepackage{hyperref}
\usepackage{babel}
\usepackage{amscd}
\usepackage{amsfonts}
\usepackage{amsthm}
\usepackage{mathrsfs}
\usepackage{dsfont}
\usepackage{mathtools}
\usepackage{enumerate}
\usepackage{bm}
\usepackage{bbm}
\usepackage{authblk}
\usepackage[inline]{enumitem}

\usepackage[english=american]{csquotes}

\usepackage{cite}

\numberwithin{equation}{section}

\DeclareMathOperator{\Tr}{Tr}
\DeclareMathOperator{\OpW}{Op_\hbar^w}
\DeclareMathOperator{\supp}{supp}

\DeclareMathOperator{\dist}{dist}

\DeclareMathOperator{\im}{\mathrm{Im}}

\newcommand{\norm}[1]{\lVert #1 \rVert}
\newcommand{\abs}[1]{\left| #1 \right|}

\newcommand{\R}{\mathbb{R}}
\newcommand{\C}{\mathbb{C}}
\newcommand{\N}{\mathbb{N}}

\newtheorem{thm}{Theorem}[section]
\newtheorem{lemma}[thm]{Lemma}
\newtheorem{proposition}[thm]{Proposition}

\theoremstyle{definition}
\newtheorem{definition}[thm] {Definition}
\newtheorem{assumption}[thm]{Assumption}

\theoremstyle{remark}
\newtheorem{remark}[thm]{Remark}

\begin{document}

\title{Sharp semiclassical spectral asymptotics for Schr\"odinger operators with non-smooth potentials}

\author{S{\o}ren Mikkelsen}
  \affil{\small{Department of Mathematics and Statistics, University of Helsinki \\ Helsinki, Finland \\ email:~\texttt{soren.mikkelsen@helsinki.fi} }}

\maketitle

\begin{abstract}
We consider semiclassical Schr\"odinger operators acting in $L^2(\R^d)$ with $d\geq3$. For these operators we establish sharp spectral asymptotics without full regularity. For the counting function we assume the potential is locally integrable and that the negative part of the potential minus a constant is one time differentiable and the derivative is H\"older continues with parameter $\mu\geq1/2$. Moreover, we also consider sharp Riesz means of order $\gamma$ with $\gamma\in(0,1]$. Here we assume the potential is locally integrable and that the negative part of the potential minus a constant is two time differentiable and the second derivative is H\"older continues with parameter $\mu$ that depends on $\gamma$.   
\end{abstract}

 \section{Introduction}
Consider a semiclassical Schr\"odinger operator $H_\hbar= - \hbar^2\Delta + V$ acting in $L^2(\R^d)$, where  $-\Delta$ is the positive Laplacian and $V$ is the potential. For the Schr\"odinger operator $H_\hbar$ the Weyl law states that
 \begin{equation}\label{EQ:weyl_general}
 	 \Tr\big[\boldsymbol{1}_{(-\infty,0]}(H_\hbar)\big] = \frac{1}{(2\pi\hbar)^d} \int_{\R^{2d}} \boldsymbol{1}_{(-\infty,0]}(p^2+V(x)) \, dpdx + o(\hbar^{-d}),
 \end{equation}
 where $\boldsymbol{1}_{\Omega}(t)$ is the characteristic function of the set $\Omega$. It has recently been proven by Frank \cite{MR4651285} that \eqref{EQ:weyl_general} is valid under the condition that $d\geq3$, $V\in L^1_{loc}(\R^d)$ and $V_{-}\in L^{\frac{d}{2}}(\R^d)$, where $V_{-} = \max(0,-V)$. These conditions are the minimal conditions such that both sides of the equality are well defined and finite. For a brief historical description of the development on establishing \eqref{EQ:weyl_general} under minimal assumptions see the introduction of \cite{MR4651285}. 
 
Under additional assumptions on the potential $V$ it was established by Helffer and Robert in \cite{MR724029} that
 \begin{equation}\label{EQ:weyl_general_1}
  \Tr\big[\boldsymbol{1}_{(-\infty,0]}(H_\hbar)\big] = \frac{1}{(2\pi\hbar)^d} \int_{\R^{2d}} \boldsymbol{1}_{(-\infty,0]}(p^2+V(x)) \, dpdx + \mathcal{O}(\hbar^{1-d})
 \end{equation}
 for all $\hbar\in(0,\hbar_0]$, $\hbar_0$ suffciently small. They proved this under the condition that $V\in C^\infty(\R^d)$, satisfies some regularity condition at infinity and $V(x)\geq c>0$ for all $x \in \Omega^c$, where $\Omega\subset \R^d$ is some open bounded set. Moreover, they assumed a non-critical condition on the energy surface $\{(x,p)\in \R^{2d} \,|\, p^2 + V(x) =0\}$. The non-critical condition can afterwards be removed see e.g. \cite{MR1343781}. The error estimate in  \eqref{EQ:weyl_general_1} is the best generic error estimate one can obtain. As an example, one can consider the operator $H_\hbar= - \hbar^2\Delta + x^2 - \lambda$, for some $\lambda>0$. For this operator we can explicitly find all eigenvalues and check by hand that \eqref{EQ:weyl_general_1} is valid with an explicit error of order $\hbar^{1-d}$.
  
When comparing the two results in dimensions $d\geq3$ it do raise the question: Is the formula \eqref{EQ:weyl_general_1} valid under less smoothness? Could it even be valid for all $V$ satisfying the assumptions of the result by Frank? The last part of the question seems currently out of reach to give a positive answer and to the authors knowledge there do not yet exist a counter example. However, for the first part of the question we will give positive answers.

We will in fact not just consider  the Weyl law but also Riesz means. That is for $\gamma\in[0,1]$ we will consider traces of the form
 \begin{equation}\label{traces_to_consider}
	\Tr\big[g_\gamma(H_{\hbar})\big],
\end{equation}
where the function $g_\gamma$ is given by
\begin{equation}
	g_\gamma(t) = \begin{cases}
	\boldsymbol{1}_{(-\infty,0]}(t) &\gamma=0
	\\
	(t)_{-}^\gamma &\gamma\in(0,1].
	\end{cases}
\end{equation}
 Frank also considered traces of the form \eqref{traces_to_consider} in \cite{MR4651285}. Helffer and Robert did only consider Weyl asymptotics in \cite{MR724029}, but proved the sharp estimate for Riesz means in \cite{MR1061661}. For later comparison and use we recall the exact statement of the results obtained by Frank in \cite{MR4651285}.
 \begin{thm}\label{Thm:Frank}
 Let $\gamma\geq1/2$ if $d=1$, $\gamma>0$ if $d=2$ and $\gamma\geq0$ if $d\geq3$. Let $\Omega\subset\R^d$ be an open set and let $V\in L^1_{loc}(\Omega)$ with $V_{-}\in L^{\gamma+d/2}(\Omega)$. Then
 \begin{equation}
 	  \Tr\big[g_\gamma(H_\hbar)\big] = \frac{1}{(2\pi\hbar)^d} \int_{\R^{2d}}g_\gamma(p^2+V(x))  \,dx dp +o(\hbar^{-d}) 
 \end{equation}
 as $\hbar\rightarrow 0$, where $H_\hbar = -\hbar^2\Delta +V(x)$ is considered in $L^2(\Omega)$ with Dirichlet boundary conditions.
 \end{thm}
One thing to observe here is that this theorem is also valid, when we are on bounded domains. We will only discuss the case, where the domain is the whole space $\R^d$, $d\geq3$. For results on sharp Weyl laws without full regularity and on bounded domains we refer the reader to the works by Ivrii \cite{MR1974451} and \cite[Vol 1]{ivrii2019microlocal1}.

\subsection{Sharp asymptotics}

We will set up some notation and recall a definition before we give the assumptions for our main theorem and state it. 
 \begin{definition}
Let $f:\R^d\mapsto\R$ be a measurable function. For each $\nu\in\R$ we define the set	
\begin{equation*}
	\Omega_{\nu,f} \coloneqq \big\{ x\in\R^d \,| \, f(x)<\nu \big\}.
\end{equation*}
\end{definition}
\begin{definition}
  For $k$ in $\N$ and $\mu$ in $[0,1]$ and $\Omega\subset\R^d$ open we denote by $C^{k,\mu}(\Omega)$
  the subspace of $C^{k}(\Omega)$ defined by
  \begin{equation*}
  	\begin{aligned}
    C^{k,\mu}(\Omega) = \big\{ f\in C^{k}(\Omega) \, \big| \, \exists C>0&:   |\partial_x^{\alpha} f(x) - \partial_x^{\alpha} f(y) |  \leq C |x-y|^{\mu} 
    \\
    & \forall \alpha \in\N^d \text{ with } \abs{\alpha}=k \text{ and } \forall x,y\in\Omega \big\}.
    \end{aligned}
  \end{equation*}
\end{definition}
These definitions are here to clarify notation. We are now ready to state our assumptions on the potential $V$.
\begin{assumption}\label{Assumption:local_potential}
  Let $V\in L^1_{loc}(\R^d)$ be a real function. Suppose there exists numbers $\nu>0$, $k\in\N_0$ and $\mu\in[0,1]$ such that the set   $\Omega_{4\nu,V}$ is open and bounded and $V\in C^{k,\mu}(\Omega_{4\nu,V})$.
  \end{assumption}
With our assumptions on the potential $V$ in place we can now state the main theorem.
 \begin{thm}\label{Thm:Main}
Let $H_{\hbar} = -\hbar^2 \Delta +V$ be a Schr\"odinger operator acting in $L^2(\R^d)$ and let $\gamma\in[0,1]$. If $\gamma=0$ we assume $d\geq3$ and if $\gamma\in(0,1]$ we assume $d\geq4$. Suppose that $V$ satisfies Assumption~\ref{Assumption:local_potential} with the numbers $\nu>0$ and $k=1$ , $\mu\geq\frac{1}{2}$ if $\gamma=0$ and $k=2$ , $\mu\geq \max(\frac{3}{2}\gamma - \frac{1}{2},0)$ if $\gamma>0$.  
Then it holds that
\begin{equation}\label{EQ:THM:Local_two_derivative}
	  \Big|\Tr\big[g_\gamma(H_\hbar)\big] - \frac{1}{(2\pi\hbar)^d} \int_{\R^{2d}}g_\gamma(p^2+V(x))  \,dx dp \Big| \leq C \hbar^{1+\gamma-d}
\end{equation}
for all $\hbar$ sufficiently small. The constant $C$ depends on the number $\nu$ and the potential $V$.
\end{thm}
When comparing the assumptions for our main theorem and Theorem~\ref{Thm:Frank} we have that in both we assume the potential to be in $L^1_{loc}(\R^d)$. But in Theorem~\ref{Thm:Frank} the additional assumptions on the potential is on the negative part of $V$, whereas we need to assume regularity for the negative part of $V-4\nu$ for some $\nu>0$. One could have hoped to only have an assumption on the negative part of $V$. However, this does not seem obtainable with the methods we use here. Firstly, the way we prove the theorem require us to have control of the potential just outside the classical allowed region ($\{x\in\R^d \, | \, V(x) \leq 0\}$). Secondly, we have that the constant in \eqref{EQ:THM:Local_two_derivative} will diverge to infinity as $\nu$ tends to zero. Hence, we cannot hope to do an approximation argument.    

The assumptions on dimensions are needed to ensure integrability of some integrals. In the case of Theorem~\ref{Thm:Frank}  there are counter examples to Weyl asymptotics for $V\in L^{\frac{d}{2}}(\R^d)$ for $d=1,2$ for details see \cite{MR1399202,MR1491550}.    

This is not the first work considering sharp Weyl laws without full regularity. The first results in a semicalssical setting was obtained by Ivrii in \cite{MR1807155}, where he also considered higher order differential operators acting in $L^2(M)$, where $M$ is a compact manifold without boundary. In this work the coefficients are assumed to be differentiable and with a H\"older continuous first derivative. This was a generalisation of works by Zielinski who previously had obtained sharp Weyl laws in high energy asymptotics in  \cite{MR1736710,MR1612880,MR1620550,MR1635856}. The results by Ivrii was generalised by Bronstein and Ivrii in \cite{MR1974450}, where they reduced the assumptions further by assuming the first derivative to have modulus of continuity $\mathcal{O}(|\log(x-y)|^{-1})$ and then again by Ivrii in \cite{MR1974451} to also include boundaries and removing the non-critical condition. The non-critical condition, used in cases without full regularity, for a semiclassical pseudo-differential operator $\OpW(a)$ is
\begin{equation}\label{EQ:non-critical_int}
	|\nabla_p a(x,p)|\geq c >0 \qquad\text{for all $(x,p)\in a^{-1}(\{0\})$}.
\end{equation}
In \cite{MR2105486} Zielinski  considers the semiclassical setting with differential operators acting in $L^2(\R^d)$ and proves an optimal Weyl Law under the assumption that all coefficients are one time differentiable with a H\"older continuous derivative. Moreover, it is assumed that the coefficients and the derivatives are bounded. In \cite{MR2105486} it is remarked that it should is possible to consider unbounded coefficients in a framework of tempered variation models. This was generalised by the author in \cite{mikkelsen2022optimal} to allow for the coefficients to be unbounded. Moreover, more general operators where also considered in \cite{mikkelsen2022optimal}. Both of these works assumed a non-critical condition \eqref{EQ:non-critical_int}. This assumption makes the results of \cite{mikkelsen2022optimal} and \cite{MR2105486} not valid for Schr\"odinger operators. Since the assumption is equivalent to assuming that 
\begin{equation}\label{EQ:non-critical_int_2}
	|V(x)|\geq c >0 \qquad\text{for all $x\in \R^d$}.
\end{equation}
The author recently established sharp local spectral asymptotics for magnetic Schr\"odinger operators in \cite{mikkelsen2023sharp}. The techniques used to establish those will be crucial for the results obtained here. The assumptions we make on regularity here is ``lower'' than the regularity assumptions made in \cite{mikkelsen2023sharp}. 

The results obtained by  Bronstein and Ivrii \cite{MR1974450} and Ivrii \cite{MR1974451,MR1807155} do assume less regularity than we do in the present work. However, the techniques used in these works seem to not translate well to a non-compact setting. 

\subsection{Non sharp asymptotics}

The methods we use to establish Theorem~\ref{Thm:Main} can also be used in cases where we have less regularity than we assume in the statement of the theorem. However, if we assume less regularity, we cannot obtain sharp remainder estimates. The results we can obtain are in the following two theorems. 
 \begin{thm}\label{Thm:Main_2}
Let $H_{\hbar} = -\hbar^2 \Delta +V$ be a Schr\"odinger operator acting in $L^2(\R^d)$ with $d\geq3$. Suppose that $V$ satisfies Assumption~\ref{Assumption:local_potential} with the number $\nu>0$, $k=1$ and $0\leq\mu\leq1$.  
Then it holds that
\begin{equation}
	  \Big|\Tr\big[g_0(H_\hbar)\big] - \frac{1}{(2\pi\hbar)^d} \int_{\R^{2d}}g_0(p^2+V(x))  \,dx dp \Big| \leq C \hbar^{\kappa-d}
\end{equation}
for all $\hbar$ sufficiently small, where $\kappa = \min[\frac{2}{3}(1+\mu),1]$. The constant $C$ depends on the number $\nu$ and the potential $V$.
\end{thm}
One can see that for $\mu\geq\frac{1}{2}$ we are in the setting of Theorem~\ref{Thm:Main} and recover the sharp estimate. For the cases where $\mu<\frac{1}{2}$ we cannot currently get optimal error. However, the ``worst'' error we can obtain is $\hbar^{\frac{2}{3}-d}$. This is still a significant improving of the estimate $\hbar^{-d}$. Moreover, since a globally Lipschitz function is almost everywhere differentiable we can with these methods obtain the error $\hbar^{\frac{2}{3}-d}$, when the potential $V$ satisfies Assumption~\ref{Assumption:local_potential} with the number $\nu>0$, $k=0$ and $\mu=1$. The author believes that this case should also have sharp estimates. 
 \begin{thm}\label{Thm:Main_3}
Let $H_{\hbar} = -\hbar^2 \Delta +V$ be a Schr\"odinger operator acting in $L^2(\R^d)$ with $d\geq4$ and let $\gamma\in(0,1]$. Suppose that $V$ satisfies Assumption~\ref{Assumption:local_potential} with the numbers $\nu>0$, $k=2$ and $0\leq \mu\leq1$.  
Then it holds that
\begin{equation}
	  \Big|\Tr\big[g_\gamma(H_\hbar)\big] - \frac{1}{(2\pi\hbar)^d} \int_{\R^{2d}}g_\gamma(p^2+V(x))  \,dx dp \Big| \leq C \hbar^{\kappa-d}
\end{equation}
for all $\hbar$ sufficiently small where $\kappa = \min[\frac{2}{3}(2+\mu),1+\gamma]$. The constant $C$ depends on the number $\nu$ and the potential $V$.
\end{thm}
Again we have that for $\mu\geq \min(\frac{3}{2}\gamma - \frac{1}{2},0)$ we again recover the sharp estimates from Theorem~\ref{Thm:Main}. When considering the result we have obtained here, we have for the case $\gamma=1$ and a $C^3$ assumption sharp error terms and even in the case of $C^2$ assumption we obtain an error of the form $\hbar^{\frac{4}{3}-d}$.

We can also obtain results for $d=2$ for the counting function and $d=2$ and $d=3$ for the Riesz means. These results will also not be sharp. These results are the content of the following two theorems.
 \begin{thm}\label{Thm:Main_4}
Let $H_{\hbar} = -\hbar^2 \Delta +V$ be a Schr\"odinger operator acting in $L^2(\R^2)$. Suppose that $V$ satisfies Assumption~\ref{Assumption:local_potential} with the number $\nu>0$, $k=1$ and $0\leq\mu\leq1$.  
Then it holds that
\begin{equation}
	  \Big|\Tr\big[g_0(H_\hbar)\big] - \frac{1}{(2\pi\hbar)^d} \int_{\R^{2d}}g_0(p^2+V(x))  \,dx dp \Big| \leq C \hbar^{\kappa-2}
\end{equation}
for all $\hbar$ sufficiently small, where $\kappa = \min[\frac{1}{3}(1+2\mu),\frac{2}{3}]$. The constant $C$ depends on the number $\nu$ and the potential $V$.
\end{thm}
 \begin{thm}\label{Thm:Main_5}
Let $H_{\hbar} = -\hbar^2 \Delta +V$ be a Schr\"odinger operator acting in $L^2(\R^d)$ with $d=2$ or $d=3$ and let $\gamma\in(0,1]$. Suppose that $V$ satisfies Assumption~\ref{Assumption:local_potential} with the numbers $\nu>0$, $k=2$ and $0\leq \mu\leq1$.  
Then it holds that
\begin{equation}
	  \Big|\Tr\big[g_\gamma(H_\hbar)\big] - \frac{1}{(2\pi\hbar)^d} \int_{\R^{2d}}g_\gamma(p^2+V(x))  \,dx dp \Big| \leq C \hbar^{\kappa-d}
\end{equation}
for all $\hbar$ sufficiently small where $\kappa = \min[\frac{1}{3}(1+2\mu+d-\gamma),\frac{1}{3}(d+2\gamma)]$. The constant $C$ depends on the number $\nu$ and the potential $V$.
\end{thm}

\subsection{Organisation of the paper}

The current paper is structured as follows. In Section~\ref{SEC:pre} we specify our notation and construct approximating/framing operators. Inspired by these framing operators we define operators that locally behave as rough Schr\"odinger operators in Section~\ref{SEC:aux}. For these operators we establish a sharp Weyl law at the end of the section. This result rely heavily on the results obtained in \cite{mikkelsen2023sharp}.  In Section~\ref{SEC:proof_main}  we first establish a result on localisations of the traces and a comparison of phase-space integrals. We end the section with a proof of the main theorems.
\subsection*{Acknowledgement}
The author is grateful to the Leverhulme Trust for their support via Research Project Grant 2020-037.
\section{Preliminaries}\label{SEC:pre}
We will for an operator $A$ acting in a Hilbert space $\mathscr{H}$ denote the operator norm by $\norm{A}_{\mathrm{op}}$ and the trace norm by $\norm{A}_1$. Moreover, we will in the following use the convention that $\N$ is the strictly positive integers and $\N_0=\N\cup\{0\}$.

Next we will describe the operators we are working with. Under
Assumption~\ref{Assumption:local_potential} we can define the operator
$H_\hbar=-\hbar^2\Delta + V$ as the Friedrichs extension of the
quadratic form given by
\begin{equation*}
  \mathfrak{h}[f,g] = \int_{\R^d} \hbar^2\sum_{i=1}^d \partial_{x_i}f(x) \overline{\partial_{x_i}g(x)}  + V(x)f(x)\overline{g(x)}\;dx, \qquad f,g \in \mathcal{D}(\mathfrak{h}),
\end{equation*}
where
\begin{equation*}
  \mathcal{D}(\mathfrak{h}) = \left\{ f\in L^2(\R^d) | \int_{\R^d} \abs{p}^2 \abs{\hat{f}(p)}^2 \;dp<\infty \text{ and } \int_{\R^d} \abs{V(x)}\abs{f(x)}^2 \;dx <\infty \right\}.
\end{equation*}
In this set up the Friedrichs extension will be unique and self-adjoint see e.g. \cite{MR0493420}.
We will in our analysis use the Helffer-Sj\"ostrand formula. Before we state it we will recall a definition of an almost analytic extension. 
\begin{definition}[Almost analytic exstension]
For $f\in C_0^\infty(\R)$ we call a function $\tilde{f} \in C_0^\infty(\C)$ an almost analytic extension if it has the properties 
\begin{equation*}
	\begin{aligned}
	|\bar{\partial} \tilde{f}(z)| &\leq C_n |\im(z)|^n, \qquad \text{for all $n\in\N_0$}
	\\
	\tilde{f}(t)&=f(t) \qquad \text{for all $t\in\R$},
	\end{aligned}
\end{equation*}
where $\bar{\partial} = \frac12 (\partial_x +i\partial_y)$.
\end{definition}
For how to construct the almost analytic extension for a given $f\in  C_0^\infty(\R)$ see e.g. \cite{MR2952218,MR1735654}. The following theorem is a simplified version of a theorem in \cite{MR1349825}.
\begin{thm}[The Helffer-Sj\"{o}strand formula]\label{THM:Helffer-Sjostrand}
  Let $H$ be a self-adjoint operator acting on a Hilbert space $\mathscr{H}$ and $f$ a function from $C_0^\infty(\R)$. Then the bounded operator $f(H)$ is given by the equation
  \begin{equation*}
    f(H) =- \frac{1}{\pi} \int_\C   \bar{\partial }\tilde{f}(z) (z-H)^{-1} \, L(dz),
  \end{equation*}
  where $L(dz)=dxdy$ is the Lebesgue measure on $\C$ and $\tilde{f}$ is an almost analytic extension of $f$.
\end{thm}
\subsection{Construction of framing operators and auxiliary asymptotics}
The crucial part in this construction is  Proposition~\ref{PRO:smoothning_of_func}, for which a proof can be found in either \cite[Proposition 1.1]{MR1974450} or \cite[Proposition 4.A.2]{ivrii2019microlocal1}. 
\begin{proposition}\label{PRO:smoothning_of_func}
  Let $f$ be in $C^{k,\mu}(\R^d)$ for a $\mu$ in $[0,1]$. Then for
  every $\varepsilon >0$ there exists a function $f_\varepsilon$ in
  $C^\infty(\R^d)$ such that
  \begin{equation}\label{EQ:pro.smoothning_of_func_bounds}
    \begin{aligned}
      \abs{\partial_x^\alpha f_\varepsilon(x) -\partial_x^\alpha f(x)
      } \leq{}& C_\alpha \varepsilon^{k+\mu-\abs{\alpha}} \qquad
      \abs{\alpha}\leq k,
      \\
      \abs{\partial_x^\alpha f_\varepsilon(x)} \leq{}& C_\alpha
      \varepsilon^{k+\mu-\abs{\alpha}} \qquad \abs{\alpha}\geq k+1,
    \end{aligned}
  \end{equation}
where $C_\alpha$ is independent of $\varepsilon$, but depends on $f$ for all $\alpha\in \N_0^d$. 
\end{proposition}
\begin{lemma}\label{LE:framing_operators}
Let $H_{\hbar} = -\hbar^2 \Delta +V$ be a Schr\"odinger operator acting in $L^2(\R^d)$ and suppose that $V$ satisfies Assumption~\ref{Assumption:local_potential} with the numbers $(\nu,k,\mu)$. Then for all $\varepsilon>0$ there exists two framing operators $H_{\hbar,\varepsilon}^{\pm}$ such that
\begin{equation}\label{LEEQ:framing_operators}
	H_{\hbar,\varepsilon}^{-} \leq H_\hbar \leq H_{\hbar,\varepsilon}^{+}
\end{equation}
in the sense of quadratic forms. The operators $H_{\hbar,\varepsilon}^{\pm}$ are explicitly given by $H_{\hbar,\varepsilon}^{\pm} = -\hbar^2 \Delta +V_\varepsilon^{\pm}$, where
\begin{equation}
	V_\varepsilon^{\pm}(x) = V^1_\varepsilon(x) +V^2(x)  \pm C\varepsilon^{k+\mu},
\end{equation}
where the function $V^1_\varepsilon(x) $ is the smooth function from Proposition~\ref{PRO:smoothning_of_func} associated to $V^1=V\varphi$ and $V^2=V(1-\varphi)$. The function $\varphi$ is chosen such that $\varphi\in C_0^\infty(\R^d)$ with $\varphi(x)=1$ for all $x\in \Omega_{3\nu,V}$ and $\supp(\varphi)\subset \Omega_{4\nu,V}$.
Moreover for all $\varepsilon>0$ sufficiently small there exists a $\tilde{\nu}>0$ such that
\begin{equation}\label{LEEQ:framing_operators2}
	\Omega_{4\tilde{\nu},V^{+}_\varepsilon} \cap \supp(V^2) = \emptyset \quad\text{and}\quad \Omega_{4\tilde{\nu},V^{-}_\varepsilon} \cap \supp(V^2) = \emptyset.
\end{equation}
\end{lemma}
\begin{proof}
We start by letting $\varphi$ be as given in the statement of the lemma and set $V^1=V\varphi$ and $V^2=V(1-\varphi)$. By assumption we have that $V^1\in C^{k,\mu}_0(\R^d)$. Hence for all $\varepsilon>0$ we get from Proposition~\ref{PRO:smoothning_of_func} the existence of $V^1_\varepsilon(x)$ such that \eqref{EQ:pro.smoothning_of_func_bounds} is satisfied with $f$ replaced by $V^1$. We now let 
\begin{equation*}
	H_{\hbar,\varepsilon}= -\hbar^2\Delta + V^1_\varepsilon + V^2.
\end{equation*}
This operator is well defined and selfadjoint since both potentials are in $L^1_{loc}(\R^d)$. Moreover we have that $H_{\hbar,\varepsilon}$ and $H_{\hbar}$ will have the same domains. Let $f\in \mathcal{D}[H_\hbar]$ we then have that
\begin{equation}\label{EQ:framing_operators}
	\begin{aligned}
	\big|\langle H_\hbar f, f  \rangle -  \langle H_{\hbar,\varepsilon} f, f  \rangle\big| &=\big| \langle (V^1 - V^1_\varepsilon) f, f  \rangle \big|
	\\
	&\leq \norm{V^1 - V^1_\varepsilon}_{L^\infty(\R^d)} \norm{f}^2_{L^2(\R^d)} \leq c\varepsilon^{k+\mu} \norm{f}^2_{L^2(\R^d)}.
	\end{aligned}
\end{equation}
From choosing a sufficiently large constant $C$ we get from \eqref{EQ:framing_operators} that by letting  $H_{\hbar,\varepsilon}^{\pm} = -\hbar^2 \Delta +V_\varepsilon^{\pm}$ with  $V_\varepsilon^{\pm}(x) = V^1_\varepsilon(x) +V^2(x)  \pm C\varepsilon^{k+\mu}$ we have that  \eqref{LEEQ:framing_operators} is satisfied with this choice of operators. 

What remains is to establish \eqref{LEEQ:framing_operators2}. We have by construction that
\begin{equation}
	\rVert V-V_\varepsilon^{\pm}\rVert_{L^\infty(\R^d)} \leq C \varepsilon^{k+\mu}.
\end{equation}
Hence if we choose $\tilde{\nu}\leq\frac{\nu}{2}$ and $\varepsilon$ is sufficiently small we can ensure that $\Omega_{4\tilde{\nu},V^{+}_\varepsilon}\subset \Omega_{3\nu,V}$ and $\Omega_{4\tilde{\nu},V^{-}_\varepsilon}\subset \Omega_{3\nu,V}$. Since we have that $\supp(V^2) \subset \Omega_{3\nu,V}^c$ by construction it follows that with such a choice of $\tilde{\nu}$ and for $\varepsilon$ sufficiently small we have  \eqref{LEEQ:framing_operators2} is true.
This concludes the proof.
\end{proof}
\begin{remark}
We will in what follows for $\varepsilon >0$ call a potential $V_\varepsilon\in C_0^\infty(\R^d)$ a rough potential of regularity $\tau\geq0$ if
\begin{equation*}
	\begin{aligned}
	\sup_{x\in\R^d} \big| \partial_x^\alpha V_\varepsilon(x)\big| \leq C_\alpha \varepsilon^{\min(0,\tau - |\alpha|)} \quad\text{for all $\alpha\in\N_0^d$},
	 \end{aligned}
\end{equation*}
where the constants $C_\alpha$ are independent of $\varepsilon$. Moreover, we  denote by a rough Schr\"odinger operator of regularity $\tau\geq0$ an operator of the form 
\begin{equation*}
	H_{\hbar,\varepsilon}=-\hbar^2\Delta + V + V_\varepsilon,
\end{equation*}
where $V\in L^1_{loc}(\R^d)$ and $V_\varepsilon$ is a rough potential of regularity $\tau$.
\end{remark}
\begin{remark}\label{RE:use_of_thm_1}
Assume we are in the setting of Lemma~\ref{LEEQ:framing_operators} then it follows from Theorem~\ref{Thm:Frank} that there exists a constant $C>0$ such that
\begin{equation*}
	\Tr\big[g_\gamma(H_{\hbar,\varepsilon}^{+} )\big]  \leq \Tr\big[g_\gamma(H_\hbar )\big] \leq \Tr\big[g_\gamma(H_{\hbar,\varepsilon}^{-}  )\big] \leq C\hbar^{-d} 
\end{equation*}
for $\hbar>0$, $\varepsilon>0$ sufficiently small. The constant $C$ only depends on the dimension, the set $\Omega_{4\nu,V}$ and $\min(V)$. The two first inequalities follows from the min-max principle. For the third inequality we can choose a potential $V^{min}$ such that
\begin{equation*}
	V^{min}(x) = 
	\begin{cases}
	 \min(V) - 1 &\text{if $x\in \Omega_{4\nu,V}$}
	 \\
	 0 & \text{if $x\notin \Omega_{4\nu,V}$}.
	\end{cases}
\end{equation*}
Then when we consider the operator $H_\hbar^{min}=-\hbar^2 \Delta + V^{min}$, defined as a Friedrichs extension of the associated form, we have that
\begin{equation*}
	H_\hbar^{min}\leq H_{\hbar,\varepsilon}^{-}
\end{equation*}
in the sense of quadratic forms. Hence using the min-max principle and Theorem~\ref{Thm:Frank} we obtain that
\begin{equation}
	\begin{aligned}
	\Tr\big[g_\gamma(H_{\hbar,\varepsilon}^{-}  )\big] \leq {}& \Tr\big[g_\gamma(H_\hbar^{min}  )\big]
	\\
	\leq {}&  \frac{1}{(2\pi\hbar)^d} \int_{\R^{2d}}g_\gamma(p^2+V^{min}(x))  \,dx dp + \tilde{C} \hbar^{-d} \leq C\hbar^{-d},
	\end{aligned}
\end{equation}
where the constant $C$ only depends on the dimension, the set $\Omega_{4\nu,V}$ and $\min(V)$.
\end{remark}
\section{Auxiliary results and model problem}\label{SEC:aux}
Inspired by the form of the framing operators we will make the following assumption. This assumption is essentially the assumption that appears in \cite{MR1343781} but with a rough potential and no magnetic field.
\begin{assumption}\label{Assumption:local_potential_1}
 Let $\mathcal{H}_{\hbar,\varepsilon}$ be an operator acting in $L^2(\R^d)$, where $\hbar,\varepsilon>0$. Suppose that
\begin{enumerate}[label={$\roman*)$}]
  \item\label{G.L.1.1}  $\mathcal{H}_{\hbar,\varepsilon}$ is self-adjoint and lower semibounded.
  \item\label{G.L.1.2}  Suppose there exists an open set $\Omega\subset\R^d$ and a rough potential $V_\varepsilon\in C^{\infty}_0(\R^d)$ of regularity $\tau\geq0$ such that $C_0^\infty(\Omega)\subset\mathcal{D}(\mathcal{H})$ and
\begin{equation*}
	\mathcal{H}_{\hbar,\varepsilon} \varphi = H_{\hbar,\varepsilon}\varphi \quad\text{for all $\varphi\in C_0^\infty(\Omega)$},
\end{equation*}
where $H_{\hbar,\varepsilon}= -\hbar^2\Delta  + V_\varepsilon$.
  \end{enumerate}
\end{assumption}
For these operators we will establish our model problem.
The first auxiliary result we will need was established in \cite{mikkelsen2023sharp}, where it is Lemma~{4.6}. It is almost the full model problem except we consider only the operator $H_{\hbar,\varepsilon}$ and not the general operator  $\mathcal{H}_{\hbar,\varepsilon}$.
\begin{lemma}\label{LE:Aux_weyl_law}
Let $\gamma\in[0,1]$ and $H_{\hbar,\varepsilon} = -\hbar^2\Delta +V_\varepsilon $ be a rough Schr\"odinger operator acting in $L^2(\R^d)$ of regularity $\tau\geq1$ if $\gamma=0$ and regularity $\tau\geq2$ if $\gamma>0$ with $\hbar\in(0,\hbar_0]$, $\hbar_0$ sufficiently small. Assume that $V_\varepsilon \in C_0^\infty(\R^d)$ and there exists a $\delta\in(0,1]$ such that $\varepsilon\geq\hbar^{1-\delta}$. Suppose there is an open set $\Omega \subset \supp(V_\varepsilon)$ and a $c>0$ such that 
\begin{equation*}
	|V_\varepsilon(x)| +\hbar^{\frac{2}{3}} \geq c \qquad\text{for all $x\in \Omega$}.
\end{equation*}
Then for $\varphi\in C_0^\infty(\Omega)$ it holds that
\begin{equation*}
	  \Big|\Tr\big[\varphi g_\gamma(H_{\hbar,\varepsilon})\big] - \frac{1}{(2\pi\hbar)^d} \int_{\R^{2d}}g_\gamma(p^2+V_\varepsilon(x))\varphi(x)  \,dx dp \Big| \leq C \hbar^{1+\gamma-d},
\end{equation*}
where the constant $C$ depends only on the dimension and the numbers $\gamma$,  $\norm{\partial^\alpha \varphi}_{L^\infty(\R^d)}$ and $\varepsilon^{-\min(0,\tau-|\alpha|)}\norm{ \partial^\alpha V_\varepsilon}_{L^\infty(\Omega)}$ for all $\alpha\in N_0^d$.
\end{lemma}
What remains in order for us to be able to prove our model problem is to establish that $\Tr\big[\varphi g_\gamma(H_{\hbar,\varepsilon})\big] $ and $\Tr\big[\varphi g_\gamma(\mathcal{H}_{\hbar,\varepsilon})\big] $ are close. To do this we will need some additional notation and results.
\begin{remark}\label{RE:propagator}
In order to prove Lemma~\ref{LE:Aux_weyl_law} as done in \cite{mikkelsen2023sharp} one needs to understand the Schr\"odinger propagator $e^{i\hbar^{-1}t H_{\hbar,\varepsilon}}$ associated to $H_{\hbar,\varepsilon}$. Under the assumptions of the lemma we can find an operator with an explicit kernel that locally approximate $e^{i\hbar^{-1}t H_{\hbar,\varepsilon}}$ in a suitable sense. This local construction is only valid for times of order $\hbar^{1-\frac{\delta}{2}}$. But if we locally have a non-critical condition the approximation can be extended to a small time interval $[-T_0,T_0]$, where $T_0$ is independent of $\hbar$. For further details see \cite{mikkelsen2022optimal}. In the following we will reference this remark and the number $T_0$. 
\end{remark}
\begin{remark}\label{RE:mollyfier_def}
Let $T\in(0,\min(T_0,T_0')]$ and $\hat{\chi}\in C_0^\infty((-T,T))$ be a real valued function such that $\hat{\chi}(s)=\hat{\chi}(-s)$ and $\hat{\chi}(s)=1$ for all $t\in(-\frac{T}{2},\frac{T}{2})$. Here $T_0$ is the number from Remark~\ref{RE:propagator} and $T_0'$ is the number from Remark~\ref{remark_small_T}. Define
\begin{equation*}
	\chi_1(t) = \frac{1}{2\pi} \int_{\R} \hat{\chi}(s) e^{ist} \,ds.
\end{equation*}
We assume that $\chi_1(t)\geq 0$ for all $t\in\R$ and there exist $T_1\in(0,T)$ and $c>0$ such that $\chi_1(t)\geq c$ for all $t\in[-T_1,T_1]$. We can guarantee these assumptions by (possible) replacing $\hat{\chi}$ by $\hat{\chi}*\hat{\chi}$. We will by $\chi_\hbar(t)$ denote the function
\begin{equation*}
	\chi_\hbar(t) = \tfrac{1}{\hbar} \chi_1(\tfrac{t}{\hbar}).
\end{equation*}
Moreover for any function $g\in L^1_{loc}(\R)$ we will use the notation
\begin{equation*}
	g^{(\hbar)}(t) =g*\chi_\hbar(t) = \int_\R g(s) \chi_\hbar(t-s).
\end{equation*}
\end{remark}
Before we proceed we will just recall the following classes of functions. These did first appear in \cite{MR1272980}.
\begin{definition}
A function $g\in C^\infty(\R\setminus\{0\})$ is said to belong to the class $C^{\infty,\gamma}(\R)$, $\gamma\in[0,1]$, if $g\in C(\R)$ for $\gamma>0$, for some constants $C >0$ and $r>0$ it holds that
\begin{equation*}
	\begin{aligned}
	g(t) &= 0, \qquad\text{for all $t\geq C$}
	\\
	|\partial_t^m g(t)| &\leq C_m |t|^r, \qquad\text{for all $m\in\N_0$ and $t\leq -C$}
	\\
	|\partial_t^m g(t)| &\leq
	\begin{cases} 
		C_m  & \text{if $\gamma=0,1$} \\  
		C_m|t|^{\gamma-m} &\text{if $\gamma\in(0,1)$} 
	\end{cases}, 
	\qquad\text{for all $m\in\N$ and $t\in [ -C,C]\setminus\{0\}$}.
	\end{aligned}
\end{equation*}
A function $g$ is said to belong to $C^{\infty,\gamma}_0(\R)$ if $g\in C^{\infty,\gamma}(\R)$ and $g$ has compact support.
\end{definition}
With this notation sat up we recall the following Tauberian type result.  This result can be found in \cite{MR1343781}, where it is Proposition~{2.8}.
\begin{proposition}\label{PRO:Tauberian}
Let $A$ be a selfadjoint operator acting in a Hilbert space $\mathscr{H}$ and $g\in C_{0}^{\infty,\gamma}(\R)$. Let $\chi_1$ be defined as in Remark~\ref{RE:mollyfier_def}. If for a Hilbert-Schmidt operator $B$
\begin{equation}\label{EQ:PRO:Tauberian_1}
	\sup_{t\in \mathcal{D}(\delta)} \norm { B^{*}\chi_\hbar (A-t) B}_1 \leq Z(\hbar),
\end{equation}
where $\mathcal{D}(\delta) = \{ t \in\R \,|\, \dist(\supp(g)),t)\leq \delta \}$, $Z(\beta)$ is some positive function and strictly positive number $\delta$. Then it holds that
\begin{equation}
	\norm { B^{*}(g(A)-g^{(\hbar)}(A)) B}_1 \leq C \hbar^{1+\gamma} Z(\hbar) + C_{N}' \hbar^N \norm{B^{*}B}_1 \quad\text{for all $N\in\N$},
\end{equation}
where the constants $C$ and $C'$ depend on the number $\delta$ and the functions $g$ and $\chi_1$ only.
\end{proposition}
The following two lemmas can be found in \cite{mikkelsen2023sharp}, where it is Lemma~{4.3} and Lemma~{4.5} respectively. We have included an additional estimate in the statement of the first lemma compared to Lemma~{4.3} in \cite{mikkelsen2023sharp} . This estimate is established as a part of the proof Lemma~{4.3} in \cite{mikkelsen2023sharp}.
\begin{lemma}\label{LE:Comparision_Loc_infty} 
Let $\mathcal{H}_{\hbar,\varepsilon}$ be an operator acting in $L^2(\R^d)$ which satisfies Assumption~\ref{Assumption:local_potential_1} with the open set $\Omega$ and let   $H_{\hbar,\varepsilon} = -\hbar^2 \Delta +V_\varepsilon$ be the associated rough Schr\"odinger operator of regularity $\tau\geq1$.
Assume that $\hbar\in(0,\hbar_0]$, with $\hbar_0$ sufficiently small.
Then for $f\in C_0^\infty(\R)$ and $\varphi\in C_0^\infty(\Omega)$ we have for any $N\in\N_0$ that
\begin{align}
	\norm{\varphi[f(\mathcal{H}_{\hbar,\varepsilon})-f(H_{\hbar,\varepsilon})]}_1 &\leq C_N \hbar^N, \label{EQLE:Comparision_Loc_infty_1}
\\
     \big \lVert \varphi [(z-\mathcal{H}_{\hbar,\varepsilon})^{-1}-(z-H_{\hbar,\varepsilon})^{-1}] \big\rVert_1 &\leq C_N \frac{\langle z \rangle^{N+\frac{d+1}{2}} \hbar^{2N-d}}{|\im(z)|^{2N+2}} , \qquad\text{$z\in \C\setminus\R$} \label{EQLE:Comparision_Loc_infty_2}
\shortintertext{and} 
	\norm{\varphi f(\mathcal{H}_{\hbar,\varepsilon})}_1&\leq C \hbar^{-d},\label{EQLE:Comparision_Loc_infty_3}
\end{align}	
The constant $C_N$ depends on $\supp(f)$, the numbers $N$, $\norm{f}_{L^\infty(\R)}$, $\norm{\partial^\alpha\varphi}_{L^\infty(\R^d)}$ for all $\alpha\in\N_0^d$.
\end{lemma}
\begin{lemma}\label{LE:assump_est_func_loc}
Let $H_{\hbar,\varepsilon} = -\hbar^2\Delta +V_\varepsilon $ be a rough Schr\"odinger operator acting in $L^2(\R^d)$ of regularity $\tau\geq1$ with $\hbar\in(0,\hbar_0]$, $\hbar_0$ sufficiently small. Assume that $V_\varepsilon \in C_0^\infty(\R^d)$ and there exists a $\delta\in(0,1]$ such that $\varepsilon\geq\hbar^{1-\delta}$. Suppose there is an open set $\Omega \subset \supp(V_\varepsilon)$ and a $c>0$ such that 
\begin{equation*}
	|V_\varepsilon(x)| +\hbar^{\frac{2}{3}} \geq c \qquad\text{for all $x\in \Omega$}.
\end{equation*}
Let $\chi_\hbar(t)$ be the function from Remark~\ref{RE:mollyfier_def}, $f\in C_0^\infty(\R)$ and $\varphi\in C_0^\infty(\Omega)$ then it holds for $s\in\R$ that
\begin{equation*}
	\norm{\varphi f(H_{\hbar,\varepsilon}) \chi_\hbar(H_{\hbar,\varepsilon}-s)f(H_{\hbar,\varepsilon})  \varphi }_1  \leq C \hbar^{-d}.
\end{equation*}
The constant $C$ depends only on the dimension, $\supp(f)$ and the numbers $\norm{f}_{L^\infty(\R)}$, $\norm{\partial^\alpha\varphi}_{L^\infty(\R)}$ for all $\alpha\in\N^d_0$ and $\varepsilon^{-\min(0,\tau-|\alpha|)}\norm{ \partial^\alpha V_\varepsilon}_{L^\infty(\Omega)}$ for all $\alpha\in\N_0^d$.
\end{lemma}
The following two lemmas are similar to Lemma~{4.8} and Lemma~{4.9} from \cite{mikkelsen2023sharp}. The proofs are analogous to the proofs and hence we will omit them here.  
\begin{lemma}\label{LE:Func_moll_com_model_reg}
Let $\mathcal{H}_{\hbar,\varepsilon}$ be an operator acting in $L^2(\R^d)$ satisfying Assumption~\ref{Assumption:local_potential_1} with the open set $\Omega$ and let   $H_{\hbar,\varepsilon} = -\hbar^2 \Delta +V_\varepsilon$ be the associated rough Schr\"odinger operator of regularity $\tau\geq1$. Assume that $\hbar\in(0,\hbar_0]$, with $\hbar_0$ sufficiently small and there exists a $\delta\in(0,1]$ such that $\varepsilon\geq\hbar^{1-\delta}$
Moreover, let $\chi_\hbar(t)$ be the function from Remark~\ref{RE:mollyfier_def}, $f\in C_0^\infty(\R)$ and $\varphi\in C_0^\infty(B(\Omega))$ then it holds for $s\in\R$ and $N\in\N$ that
\begin{equation} \label{EQLE:Func_moll_com_model_reg}
	\norm{\varphi f(\mathcal{H}_{\hbar,\varepsilon}) \chi_\hbar(\mathcal{H}_{\hbar,\varepsilon}-s)f(\mathcal{H}_{\hbar,\varepsilon})  \varphi -\varphi f(H_{\hbar,\varepsilon}) \chi_\hbar(H_{\hbar,\varepsilon}-s)f(H_{\hbar,\varepsilon})  \varphi }_1 \leq C_N \hbar^N.
\end{equation}
Moreover, suppose there exists some $c>0$ such that 
\begin{equation*}
	|V_\varepsilon(x)| +\hbar^{\frac{2}{3}} \geq c \qquad\text{for all $x\in \Omega$}.
\end{equation*}
Then it holds that
\begin{equation}\label{LEEQ:Func_moll_com_model_reg}
	\norm{\varphi f(\mathcal{H}_{\hbar,\varepsilon}) \chi_\hbar(\mathcal{H}_{\hbar,\varepsilon}-s)f(\mathcal{H}_{\hbar,\varepsilon})  \varphi }_1  \leq C \hbar^{-d}.
\end{equation}
The constants $C_N$ and $C$ depend only on the dimension and the numbers $\norm{f}_{L^\infty(\R)}$, $\norm{\partial^\alpha\varphi}_{L^\infty(\R)}$ and $\varepsilon^{-\min(0,\tau-|\alpha|)}\norm{ \partial^\alpha V_\varepsilon}_{L^\infty(\Omega)}$ for all $\alpha\in\N_0^d$.
\end{lemma}
\begin{remark}\label{remark_small_T}
In the proof of Lemma~\ref{LE:Func_moll_com_model_reg} one will need that the function $\chi_1$ is supported in a sufficiently small interval around zero. The reason for this is that in the proof we need an uniform estimate in time of the following type 
\begin{equation*}
	\norm{\OpW(1-\theta_2) e^{it\hbar^{-1} H_{\hbar,\varepsilon} } f(H_{\hbar,\varepsilon}) \OpW(\theta_1)}_{\mathrm{op}} \leq C_N \hbar^{N} \qquad\text{ for all $N\in\N$}.
\end{equation*}
Here $\theta_1,\theta_2\in C^\infty_0(\R^{2d})$ such that
\begin{equation*}
	\dist\big\{\supp(\theta_1),\supp(1-\theta_2)\big\} \geq c>0.
\end{equation*}	
That such an estimate is indeed true uniformly for $|t|\leq T_0'$ follows from Lemma~{4.7} from \cite{mikkelsen2023sharp}, where $T_0'$ is some positive fixed time. For the proof of  Lemma~\ref{LE:Func_moll_com_model_reg} we have ensured that the support of $\chi_1$ is sufficiently small by assumptions made on the function in Remark~\ref{RE:mollyfier_def}. There is also a version of this result in \cite{MR1272980}, where it is Lemma~{5.1}. This result is stated and proves for semiclassical pseudo-differential operators.     
\end{remark}
\begin{lemma}\label{LE:trace_com_model_reg}
Let $\gamma\in[0,1]$ and $\mathcal{H}_{\hbar,\varepsilon}$ be an operator acting in $L^2(\R^d)$. Suppose  $\mathcal{H}_{\hbar,\varepsilon}$ satisfies Assumption~\ref{Assumption:local_potential_1} with the open set $\Omega$ and let   $H_{\hbar,\varepsilon} = -\hbar^2 \Delta +V_\varepsilon$ be the associated rough Schr\"odinger operator of regularity $\tau\geq1$ if $\gamma=0$ and $\tau\geq2$ if $\gamma>0$. Assume that $\hbar\in(0,\hbar_0]$, with $\hbar_0$ sufficiently small and there exists a $\delta\in(0,1]$ such that $\varepsilon\geq\hbar^{1-\delta}$. 
Moreover, suppose there exists some $c>0$ such that 
\begin{equation*}
	|V_\varepsilon(x)| +\hbar^{\frac{2}{3}} \geq c \qquad\text{for all $x\in \Omega$}
\end{equation*}
Then for $\varphi\in C_0^\infty(\Omega)$ it holds that
\begin{equation}
	\Big|\Tr\big[\varphi g_\gamma(\mathcal{H}_{\hbar,\varepsilon})\big] -\Tr\big[\varphi g_\gamma(H_{\hbar,\varepsilon})\big] \Big|  \leq C \hbar^{1+\gamma-d} + C_N'  \hbar^{N}.
\end{equation}
The constants $C$ and $C_N'$ depend on the dimension and the numbers  $\gamma$, $\norm{\partial^\alpha\varphi}_{L^\infty(\R)}$ and $\varepsilon^{-\min(0,\tau-|\alpha|)}\norm{ \partial^\alpha V_\varepsilon}_{L^\infty(\Omega)}$ for all $\alpha\in\N_0^d$.
\end{lemma}
\begin{proof}
Since both operators are lower semi-bounded we may assume that $g$ is compactly supported. Let $f\in C_0^\infty(\R)$ such that $f(t)g_\gamma(t)= g_\gamma(t)$ for all $t\in\R$ and let $\varphi_1\in C_0^\infty(\Omega)$ such that $\varphi(x)\varphi_1(x) = \varphi(x)$ for all $x\in\R^d$. Moreover, let $\chi_\hbar(t)$ be the function from Remark~\ref{RE:mollyfier_def} and  set $g_\gamma^{(\hbar)}(t) = g_\gamma*\chi_\hbar(t)$. With this notation set up we have that
\begin{equation}\label{EQ:trace_com_model_reg_1}
	\begin{aligned}
	\MoveEqLeft \Big|\Tr\big[\varphi g_\gamma(\mathcal{H}_{\hbar,\varepsilon})\big] -\Tr\big[\varphi g_\gamma(H_{\hbar,\varepsilon})\big] \Big|
	\\
	\leq {}& \lVert \varphi \varphi_1f(\mathcal{H}_{\hbar,\varepsilon}) (g_\gamma(\mathcal{H}_{\hbar,\varepsilon}) - g_\gamma^{(\hbar)}(\mathcal{H}_{\hbar,\varepsilon}))f(\mathcal{H}_{\hbar,\varepsilon})\varphi_1] \rVert_1
	\\
	&+\lVert \varphi \varphi_1 f(H_{\hbar,\varepsilon}) (g_\gamma(H_{\hbar,\varepsilon})-g_\gamma^{(\hbar)}(H_{\hbar,\varepsilon}))f(H_{\hbar,\varepsilon})\varphi_1] \rVert_1
	+ \norm{\varphi}_{L^\infty(\R^d)} \int_{\R} g_\gamma(s)  \,ds
	\\
	&\times  \sup_{s\in\R} \lVert\varphi \varphi_1 f(\mathcal{H}_{\hbar,\varepsilon}) \chi_\hbar(\mathcal{H}_{\hbar,\varepsilon}-s)f(\mathcal{H}_{\hbar,\varepsilon})  \varphi_1 -\varphi_1 f(H_{\hbar,\varepsilon}) \chi_\hbar(H_{\hbar,\varepsilon}-s)f(H_{\hbar,\varepsilon})  \varphi_1\rVert_1.
	\end{aligned}
\end{equation}
Lemma~\ref{LE:assump_est_func_loc} and Lemma~\ref{LE:Func_moll_com_model_reg} gives us that the assumptions of Proposition~\ref{PRO:Tauberian} is fulfilled with $B$ equal to $\varphi_1 f(H_\hbar)$ and $\varphi_1 f(H_{\hbar,\varepsilon})$  respectively. Hence we hvae that 
\begin{equation}\label{EQ:trace_com_model_reg_2}
	\begin{aligned}
	 \lVert \varphi \varphi_1f(\mathcal{H}_{\hbar,\varepsilon}) (g_\gamma(\mathcal{H}_{\hbar,\varepsilon}) - g_\gamma^{(\hbar)}(\mathcal{H}_{\hbar,\varepsilon}))f(\mathcal{H}_{\hbar,\varepsilon})\varphi_1] \rVert_1
	\leq C \hbar^{1+\gamma-d}
	\end{aligned}
\end{equation}
and
\begin{equation}\label{EQ:trace_com_model_reg_3}
	\begin{aligned}
	\lVert \varphi \varphi_1 f(H_{\hbar,\varepsilon}) (g_\gamma(H_{\hbar,\varepsilon})-g_\gamma^{(\hbar)}(H_{\hbar,\varepsilon}))f(H_{\hbar,\varepsilon})\varphi_1] \rVert_1 \leq C \hbar^{1+\gamma-d}.
	\end{aligned}
\end{equation}
From applying Lemma~\ref{LE:Func_moll_com_model_reg} we get for all $N\in\N$ that
\begin{equation}\label{EQ:trace_com_model_reg_4}
	\begin{aligned}
	\MoveEqLeft  \sup_{s\in\R} \lVert\varphi \varphi_1 f(\mathcal{H}_{\hbar,\mu}) \chi_\hbar(\mathcal{H}_{\hbar,\mu}-s)f(\mathcal{H}_{\hbar,\mu})  \varphi_1 -\varphi_1 f(H_{\hbar,\mu\varepsilon}) \chi_\hbar(H_{\hbar,\mu\varepsilon}-s)f(H_{\hbar,\mu,\varepsilon})  \varphi_1\rVert_1
	 \\
	 &\leq C_N\hbar^N.
	\end{aligned}
\end{equation}
Finally from combining the estimates in \eqref{EQ:trace_com_model_reg_1}, \eqref{EQ:trace_com_model_reg_2}, \eqref{EQ:trace_com_model_reg_3} and \eqref{EQ:trace_com_model_reg_4}we obtain the desired estimate and this concludes the proof.
\end{proof}
For operators that satisfies Assumption~\ref{Assumption:local_potential_1} we can establish the following model theorem. The proof of the theorem is similar to the proof of Theorem~{5.2} in \cite{mikkelsen2023sharp}.
\begin{thm}\label{THM:Loc_mod_prob}
Let $\gamma\in[0,1]$ and $\mathcal{H}_{\hbar,\varepsilon}$ be an operator acting in $L^2(\R^d)$. Suppose  $\mathcal{H}_{\hbar,\varepsilon}$ satisfies Assumption~\ref{Assumption:local_potential_1} with the open set $\Omega$ and let   $H_{\hbar,\varepsilon} = -\hbar^2 \Delta +V_\varepsilon$ be the associated rough Schr\"odinger operator of regularity $\tau\geq1$ if $\gamma=0$ and $\tau\geq2$ if $\gamma>0$. Assume that $\hbar\in(0,\hbar_0]$, with $\hbar_0$ sufficiently small and there exists a $\delta\in(0,1]$ such that $\varepsilon\geq\hbar^{1-\delta}$. Moreover, suppose there exists some $c>0$ such that 
\begin{equation}\label{THM:model_prob_global_Non_crit}
	|V_\varepsilon(x)| +\hbar^{\frac{2}{3}} \geq c \qquad\text{for all $x\in\Omega$}.
\end{equation}
Then for any $\varphi\in C_0^\infty(\Omega)$ it holds that
\begin{equation*}
	  \Big|\Tr\big[\varphi g_\gamma(\mathcal{H}_{\hbar,\varepsilon})\big] - \frac{1}{(2\pi\hbar)^d} \int_{\R^{2d}}g_\gamma(p^2+V_\varepsilon(x)) \varphi(x) \,dx dp \Big| \leq C \hbar^{1+\gamma-d},
\end{equation*}
where the constant $C$ depends only on the dimension and the numbers $\norm{\partial^\alpha \varphi}_{L^\infty(\R^d)}$ and $\varepsilon^{-\min(0,\tau-|\alpha|)}\norm{ \partial^\alpha V_\varepsilon}_{L^\infty(\Omega)}$ for all $\alpha\in N_0^d$. 
\end{thm}
\begin{proof}
Firstly observe that under the assumptions of this theorem we have that $\mathcal{H}_{\hbar,\varepsilon}$ and $H_{\hbar,\varepsilon}$ satisfies the assumptions of Lemma~\ref{LE:trace_com_model_reg}. Moreover, we have that $H_{\hbar,\varepsilon}$ satisfies the assumptions of Lemma~\ref{LE:Aux_weyl_law}. Hence from applying Lemma~\ref{LE:trace_com_model_reg} and Lemma~\ref{LE:Aux_weyl_law} we obtain that
\begin{equation*}
	\begin{aligned}
	  \MoveEqLeft \Big|\Tr\big[\varphi g_\gamma(\mathcal{H}_{\hbar,\varepsilon})\big] - \frac{1}{(2\pi\hbar)^d} \int_{\R^{2d}}g_\gamma(p^2+V_\varepsilon(x)) \varphi(x) \,dx dp \Big| 
	  \\
	  \leq{}& \Big|\Tr \big[\varphi g_\gamma(\mathcal{H}_{\hbar,\varepsilon})]- \varphi g_\gamma(H_{\hbar,\varepsilon})\big] \Big| + \Big|\Tr\big[\varphi g_\gamma(H_{\hbar,\varepsilon})\big] - \frac{1}{(2\pi\hbar)^d} \int_{\R^{2d}}g_\gamma(p^2+V_\varepsilon(x)) \varphi(x) \,dx dp \Big| 
	  \\
	  \leq{}& C \hbar^{1+\gamma-d}.
	  \end{aligned}
\end{equation*}
This concludes the proof.
\end{proof}

\section{Towards a proof of the main theorem}\label{SEC:proof_main}
At the end of this section we will prove our main theorem. But before we do this we will first prove some lemmas that we will need in the proof. The first is a lemma that allow us to localise the trace we consider. The second one is a comparison of phase space integrals.
\subsection{Localisation of traces and comparison of phase-space integrals}
Before we state the lemma on localisation of the trace we recall the following Agmon type estimate from \cite{MR4182014}, where it is Lemma~{A.1}.
\begin{lemma}\label{LE:Agmon_type_lem}
Let $H_\hbar=-\hbar^2\Delta +V$ be a Schr\"odinger operator acting in $L^2(\R^d)$, where $V$ is in $L^1_{loc}(\R^d)$ and suppose that there exist an $\nu>0$ and a open bounded sets $U$ such that
\begin{equation}\label{EQ:Agmon_type_lem}
   	V(x) \geq {\nu} \quad \text{when } x\in U^c.
\end{equation}
Let $d(x) = \dist(x,U_a)$, where
\begin{equation*}
	U_a = \{ x\in\R^d \, |\, \dist(x,U)<a\} 
\end{equation*}
and let $\psi$ be a normalised solution to the equation
\begin{equation*}
    	H_\hbar\psi=E\psi,
\end{equation*}
with $E<\nu/4$. Then there exists a $C>0$ such that
\begin{equation*}
    	\norm{e^{\delta \hbar^{-1} d } \psi }_{L^2(\R^d)} \leq C,
\end{equation*}
for $\delta=\tfrac{\sqrt{\nu}}{8}$. The constant $C$ depends on $a$ and is uniform in $V$, $\nu$ and $U$ satisfying \eqref{EQ:Agmon_type_lem}.
\end{lemma}
In the formulation of the lemma, we have presented here we consider $U_a$ for $a>0$ and not just $U_1$ as in \cite{MR4182014}. There are no difference in the proof. However, one thing to remark is that the constant $C$ will diverge to infinity as $a$ tends to $0$. Moreover, we have in the statement highlighted the uniformity of the constant in the potential $V$, the number $\nu$ and the set $U$. That the constant is indeed uniform in these follows directly from the proof given in \cite{MR4182014}.
\begin{lemma}\label{LE:localise_trace}
Let $\gamma\in[0,1]$ and $H_\hbar=-\hbar^2\Delta +V$ be a Schr\"odinger operator acting in $L^2(\R^d)$, where $V$ is in $L^1_{loc}(\R^d)$ and suppose that there exist an $\nu>0$ and a open bounded sets $U$ such that $V(x)\boldsymbol{1}_U(x)\in L^{\gamma+\frac{d}{2}}(\R^d)$ and
\begin{equation}\label{EQLE:localise_trace}
   	V(x) \geq {\nu} \quad \text{when } x\in U^c.
\end{equation}
Fix $a>0$ and let $\varphi\in C_0^\infty(\R^d)$ such that $\varphi(x)=1$ for all $x\in U_a$, where
\begin{equation*}
	U_a = \{ x\in\R^d \, |\, \dist(x,U)<a\}. 
\end{equation*}
Then for every $N\in\N$ it holds that
\begin{equation*}
	\Tr \big[g_\gamma(H_\hbar)\big] = \Tr \big[g_\gamma(H_\hbar)\varphi\big] + C_N \hbar^N,
\end{equation*}
where the constant is $C_N$ depends on $a$ and is uniform in $V$, $\nu$ and $U$ satisfying \eqref{EQLE:localise_trace}. 
\end{lemma}
\begin{proof}
Using linearity of the trace we have that
\begin{equation}\label{EQ:localise_trace}
	\Tr\big[g_\gamma( H_{\hbar}) \big]= \Tr\big[g_\gamma( H_{\hbar})\varphi \big] + \Tr\big[g_\gamma( H_{\hbar})(1-\varphi) \big].
\end{equation}
For the second term on the right hand side of \eqref{EQ:localise_trace} we calculate the trace in a normalised basis of eigenfunctions  for $ H_{\hbar}$ called  $\psi_n$ with eigenvalue $E_n$.   
\begin{equation}\label{EQ:localise_trace_2}
	 \Tr\big[g_\gamma( H_{\hbar})(1-\varphi) \big] 
	 = \sum_{E_n\leq0} \langle g_\gamma(H_{\hbar})(1-\varphi) \psi_n , \psi_n \rangle
	 = \sum_{E_n\leq0} g_\gamma(E_n) \norm{\sqrt{1-\varphi} \psi_n }_{L^2(\R^d)}^2. 
\end{equation}
To estimate the $L^2$-norms we let $d(x) = \dist(x,U_{\frac12a})$. For all $x\in \supp(1-\varphi)$ we have that $d(x)>0$ since $\varphi(x)=1$ for all $x\in U_a$. We get from Lemma~\ref{LE:Agmon_type_lem} that there exists a constant $C$ depending on $a$ such that for all normalised eigenfunctions $\psi_n$ with eigenvalue less than $\frac{\nu}{4}$ we have the estimate
\begin{equation}\label{EQ:localise_trace_3}
	\norm{e^{\tilde{\delta} \hbar^{-1} d } \psi_n }_{L^2(\R^d)} \leq C,
\end{equation}
where $\tilde{\delta}=\frac{\sqrt{\nu}}{8}$ and $C$ is uniform in $V$, $\nu$ and $U$ satisfying \eqref{EQLE:localise_trace}.  Using this estimate and the observations made for $d(x)$ we get for all norms in \eqref{EQ:localise_trace_2} and all $N\in\N$ the estimate
\begin{equation}\label{EQ:localise_trace_4}
	\begin{aligned}
	\norm{\sqrt{1-\varphi} \psi_n }_{L^2(\R^d)}^2 
	&\leq  \norm{\sqrt{1-\varphi} e^{-\tilde{\delta} \hbar^{-1} d }}_{L^\infty(\R^d)}^2 \norm{ e^{\tilde{\delta} \hbar^{-1} d }   \psi_n }_{L^2(\R^d)}^2
	\\
	&\leq C \big\lVert \sqrt{1-\varphi} \big( \tfrac{\hbar}{\tilde{\delta} d}\big)^N \big(\tfrac{\tilde{\delta} d}{\hbar}\big)^N e^{-\tilde{\delta} \hbar^{-1} d }\big\rVert_{L^\infty(\R^d)}^2 \leq C_N \hbar^{2N}.
	 \end{aligned}
\end{equation}
Combining \eqref{EQ:localise_trace_2} with the estimate obtained in \eqref{EQ:localise_trace_4} we get for all $N\in\N$ that
\begin{equation}\label{EQ:localise_trace_5}
	\begin{aligned}
	 \Tr\big[g_\gamma( H_{\hbar})(1-\varphi) \big] 
	 &\leq C_N  \hbar^{2N}  \sum_{E_n\leq0} g_\gamma(E_n) = C_N  \hbar^{2N}   \Tr\big[g_\gamma( H_{\hbar})\big] \leq  \tilde{C}_N  \hbar^{2N-d}, 
	 \end{aligned}
\end{equation}
where we in the last estimate have used Theorem~\ref{Thm:Frank}. Combining \eqref{EQ:localise_trace} and  \eqref{EQ:localise_trace_5} we obtain the desired estimate. 
\end{proof}
\begin{remark}
When we will apply the above Lemma we need to ensure the constant is the same for the two cases we consider. To ensure this we use Theorem~\ref{Thm:Frank} as descibed in Remark~\ref{RE:use_of_thm_1} at the end of the proof.
\end{remark}
The next lemma is a result on comparing phase-space integrals. Similar estimates are obtained with different methods in \cite{mikkelsen2022optimal}. These are parts of larger proofs and not an independent lemma. The following Lemma is taken from \cite{mikkelsen2023sharp}, where it is Lemma~{5.1}.
\begin{lemma}\label{LE:comparison_phase_space_int}
Suppose $\Omega\subset\R^d$ is an open set and let $\varphi\in C_0^\infty(\Omega)$. Moreover, let $\varepsilon>0$, $\hbar\in(0,\hbar_0]$ and  $V,V_\varepsilon\in L^1_{loc}(\R^d)\cap C(\Omega)$. Suppose that 
\begin{equation}\label{EQLE:comparison_phase_space_int}
	\norm{V-V_\varepsilon}_{L^\infty(\Omega)}\leq c\varepsilon^{k+\mu}.
\end{equation}
Then for $\gamma\in[0,1]$ and $\varepsilon$ sufficiently small it holds that
\begin{equation}\label{LEEQ:Loc_mod_prob_5}
	\begin{aligned}
	 \Big| \int_{\R^{2d}} [g_\gamma(p^2+V_\varepsilon(x))-g_\gamma(p^2+V(x))]\varphi(x) \,dx dp  \Big| 
	 \leq C\varepsilon^{k+\mu},
	  \end{aligned}
\end{equation}
where the constant $C$ depends on the dimension and the numbers $\gamma$ and $c$  in  \eqref{EQLE:comparison_phase_space_int}.
\end{lemma}
\subsection{Proof of main theorem}
The proof of the main theorem is based on a multi-scale argument. Before we prove the main theorem by using these techniques we will recall the following crucial lemma.
\begin{lemma}\label{LE:partition_lemma}
  Let $\Omega\subset\R^d$ be an open set and let $l$ be a function in
  $C^1(\bar{\Omega})$ such that $l>0$ on $\bar{\Omega}$ and assume
  that there exists $\rho$ in $(0,1)$ such that
  \begin{equation}
    \abs{\nabla_x l(x)} \leq \rho,
  \end{equation}
  for all $x$ in $\Omega$. 
  
  Then
  \begin{enumerate}[label=$\roman*)$]
	
  \item There exists a sequence $\{x_k\}_{k=0}^\infty$ in $\Omega$
    such that the open balls $B(x_k,l(x_k))$ form a covering of
    $\Omega$. Furthermore, there exists a constant $N_\rho$, depending only
    on the constant $\rho$, such that the intersection of more than
    $N_\rho$ balls is empty.
		
  \item One can choose a sequence $\{\varphi_k\}_{k=0}^\infty$ such
    that $\varphi_k \in C_0^\infty(B(x_k,l(x_k)))$ for all $k$ in
    $\N$. Moreover, for all multiindices $\alpha$ and all $k$ in $\N$
    \begin{equation*}
      \abs{\partial_x^\alpha \varphi_k(x)}\leq C_\alpha l(x_k)^{-{\abs{\alpha}}},
    \end{equation*} 	   
    and
    \begin{equation*}
      \sum_{k=1}^\infty \varphi_k(x) = 1,
    \end{equation*}
    for all $x$ in $\Omega$.
  \end{enumerate}
\end{lemma}
This lemma is taken from \cite{MR1343781} where it is Lemma 5.4. The proof is analogous to the proof of \cite[Theorem
1.4.10]{MR1996773}. We are now ready to prove the main theorem.

\begin{proof}[Proof of Theorem~\ref{Thm:Main}]
Let $H_{\hbar,\varepsilon}^{\pm}$ be the two framing operators constructed in Lemma~\ref{LE:framing_operators}, where we choose $\varepsilon=\hbar^{1-\delta}$. For $\gamma=0$ we choose $\delta=\frac{\mu}{1+\mu}$ and if $\gamma>0$ we choose $\delta=\frac{1+\mu-\gamma}{2+\mu}$. Note that our assumptions on $\mu$ will in all cases ensure that $\delta\geq\frac13$. Moreover, we get that
\begin{equation}\label{EQ:proof_main_0}
	\begin{aligned}
	\varepsilon^{1+\mu} &= \hbar &\text{$\gamma=0$}
	\\
	\varepsilon^{2+\mu} &= \hbar^{1+\gamma} &\text{$\gamma>0$}.
	\end{aligned}
\end{equation}
Since we have that $H_{\hbar,\varepsilon}^{-} \leq H_\hbar \leq H_{\hbar,\varepsilon}^{+}$ in the sense of quadratic forms. It follows from the min-max theorem that
\begin{equation}\label{EQ:proof_main_1}
	\Tr\big[g_\gamma( H_{\hbar,\varepsilon}^{+}) \big] \leq \Tr\big[g_\gamma( H_{\hbar}) \big] \leq \Tr\big[g_\gamma( H_{\hbar,\varepsilon}^{-}) \big].
\end{equation}
The aim is now to obtain spectral asymptotics for $\Tr\big[g_\gamma( H_{\hbar,\varepsilon}^{+}) \big]$ and $\Tr\big[g_\gamma( H_{\hbar,\varepsilon}^{-}) \big]$. The arguments will be analogous so we will drop the superscript $\pm$ for the operator $H_{\hbar,\varepsilon}^{\pm}$ in what follows. Let $\varphi\in C_0^\infty(\R^d)$ with $\varphi(x)=1$ for all $x\in \Omega_{\tilde{\nu},V_\varepsilon}$ and $\supp(\varphi)\subset \Omega_{2\tilde{\nu},V_\varepsilon}$. Then applying Lemma~\ref{LE:localise_trace} we obtain for all $N\in\N$ that
\begin{equation}\label{EQ:proof_main_2}
	\Tr\big[g_\gamma( H_{\hbar,\varepsilon}) \big]= \Tr\big[g_\gamma( H_{\hbar,\varepsilon})\varphi \big] +C_N\hbar^N.
\end{equation}
For the terms $\Tr\big[g_\gamma( H_{\hbar,\varepsilon})\varphi \big]$ we use a multiscale argument such that  we locally can apply Theorem~\ref{THM:Loc_mod_prob}.  Recall that from from Lemma~\ref{LE:framing_operators} we have that
\begin{equation*}
	V_\varepsilon(x) = V^1_\varepsilon(x) +V^2(x)  \pm C\varepsilon^{\tau+\mu},
\end{equation*}
where $\supp(V^2)\cap\Omega_{4\tilde{\nu},V_\varepsilon}^C =\emptyset$ and $V^1_\varepsilon\in C_0^\infty(\R^d)$. We let $\varphi_1\in C_0^\infty(\R^d)$ such that $\varphi_1(x) = 1$ for all  $x\in\Omega_{2\tilde{\nu},V_\varepsilon}$ and $\supp(\varphi_1)\subset \Omega_{4\tilde{\nu},V_\varepsilon}$. With this function we have that 
\begin{equation}
	\varphi_1(x)V_\varepsilon^{\pm}(x) = \varphi_1(x)(V^1_\varepsilon(x)   \pm C\varepsilon^{\tau+\mu}).
\end{equation}
Note that with these assumptions on $\varphi_1(x)$ we have that $\varphi_1(x)\varphi(x)=\varphi(x)$ for all $x\in\R^d$. This observations ensures that when we define our localisation function $l(x)$ below it is positive on the set $\supp(\varphi)$.
Before we define our localisation functions we remark that due to the continuity of $V_\varepsilon$ on $\Omega_{4\tilde{\nu},V_\varepsilon}$ there exists a number $\epsilon>0$ such that
\begin{equation*}
	\dist(\supp(\varphi), \Omega_{2\tilde{\nu},V_\varepsilon}^{c}) >\epsilon.
\end{equation*}
The number $\epsilon$ is important for our localisation functions. As we need to ensure the supports are contained in the set $\Omega_{2\tilde{\nu},V_\varepsilon}$. We let 
\begin{equation*}
	l(x) = A^{-1}\sqrt{ |\varphi_1(x)V_\varepsilon(x)|^2 + \hbar^\frac{4}{3}} \quad\text{and}\quad f(x)=\sqrt{l(x)}. 
\end{equation*} 
Where we choose $A >0$ sufficiently large such that
\begin{equation}\label{EQ:proof_main_3}
	l(x) \leq \frac{\epsilon}{9}  \quad\text{and}\quad |\nabla l(x)|\leq \rho <\frac{1}{8}
\end{equation} 
for all $x\in\overline{\supp(\varphi)}$. Note that due to our assumptions on $V_\varepsilon$ we can choose $A$ independent of $\hbar$ and uniformly for $\hbar\in(0,\hbar_0]$. Moreover, we have that
\begin{equation}\label{EQ:proof_main_4}
	|\varphi_1(x)V_\varepsilon(x)| \leq A l(x) 
\end{equation}
for all $x\in\R^d$. By Lemma~\ref{LE:partition_lemma} with the set $\supp(\varphi)$ and the function $l(x)$ there exists a sequence $\{x_k\}_{k=1}^\infty$ in $\supp(\varphi)$ such that $\supp(\varphi) \subset \cup_{k\in\N} B(x_k,l(x_k))$ and there exists a constant $N_{\frac{1}{8}}$ such that at most $N_{\frac{1}{8}}$ of the sets $B(x_k,l(x_k))$ can have a non-empty intersection. Moreover, there exists a sequence $\{\varphi_{k}\}_{k=1}^\infty$ such that $\varphi_k\in C_0^\infty(B(x_k,l(x_k)))$,
\begin{equation}\label{EQ:proof_main_5}
	\big| \partial_x^\alpha \varphi_k(x) \big| \leq C_\alpha l(x_k)^{-|\alpha|} \qquad\text{for all $\alpha\in\N_0$},
\end{equation}
and
\begin{equation*}
	\sum_{k=1}^\infty \varphi_k(x) =1  \qquad\text{for all $\supp(\varphi)$}.
\end{equation*}
We have that $ \cup_{k\in\N} B(x_k,l(x_k))$ is an open covering of $\supp(\varphi)$ and since this set is compact there exists a finite subset $\mathcal{I}'\subset \N$ such that  
\begin{equation*}
	\supp(\varphi) \subset \bigcup_{k\in\mathcal{I}'} B(x_k,l(x_k)).
\end{equation*}
In order to ensure that we have a finite partition of unity over the set $\supp(\varphi)$ we define the set
\begin{equation*}
	\mathcal{I} = \bigcup_{j\in\mathcal I'} \big\{ k\in\N \,|\, B(x_k,l(x_k))\cap B(x_j,l(x_j)) \neq \emptyset \big\}.
\end{equation*}
Then we have that $\mathcal{I} $ is still finite since at most $N_{\frac{1}{8}}$ balls can have non-empty intersection. Moreover, we have that
\begin{equation*}
	\sum_{k\in\mathcal{I}} \varphi_k(x) =1  \qquad\text{for all $\supp(\varphi)$}.
\end{equation*}
From this we get the following identity
\begin{equation}\label{EQ:proof_main_6}
	 \Tr\big[\varphi \boldsymbol{1}_{ (-\infty,0]}(H_{\hbar,\varepsilon})\big] = \sum_{k\in\mathcal{I}}  \Tr\big[\varphi_k \varphi \boldsymbol{1}_{ (-\infty,0]}(H_{\hbar,\varepsilon})\big] ,
\end{equation}
where we have used linearity of the trace. We will for the remaining part of the proof use the following notation 
\begin{equation*}
	l_k=l(x_k), \quad f_k=f(x_k) \quad h_k = \frac{\hbar}{l_kf_k} \quad\text{and}\quad \varepsilon_k = \hbar_k^{1-\delta}.
\end{equation*}
We have that $h_k$ is uniformly bounded from above since
\begin{equation*}
	l(x)f(x) = A^{-\frac32}(\abs{\varphi_1(x)V_\varepsilon(x)}^2+\hbar^{\frac43})^{\frac34} \geq A^{-\frac32} \hbar, 
\end{equation*}
for all $x$. Moreover, since we by assumption have that $\delta\geq \frac13$ and $l_k=f_k^2$ we obtain that
\begin{equation}\label{EQ:proof_main_7}
	l_k \varepsilon^{-1} \leq \varepsilon_k^{-1}.
\end{equation}
 We define the two unitary operators $U_l$ and $T_z$ by
\begin{equation*}
	U_l f(x) = l^{\frac{d}{2}} f( l x) \quad\text{and}\quad T_zf(x)=f(x+z) \qquad\text{for $f\in L^2(\R^d)$}.
\end{equation*}
Moreover we set
\begin{equation*}
	\begin{aligned}
	\tilde{H}_{\varepsilon,h_k} = f_k^{-2} (T_{x_k} U_{l_k}) H_\hbar (T_{x_k} U_{l_k})^{*}
	 = - h_k^2 \Delta +\tilde{V}_\varepsilon(x),
	\end{aligned}
\end{equation*}
where $\tilde{V}_\varepsilon(x)=f_k^{-2} V_\varepsilon(l_kx+x_k)$. We will here need to establish that this rescaled operator satisfies the assumptions of Theorem~\ref{THM:Loc_mod_prob} with $\hbar_k$, $\varepsilon_k$ and the set $B(0,8)$. To establish this we firstly observe that  by \eqref{EQ:proof_main_3} we have
\begin{equation}\label{EQ:Rough_weyl_asymptotics_3.5}
	(1-8\rho) l_k \leq l(x) \leq (1+8\rho) l_k \qquad\text{for all $x \in B(x_k,8l_k)$}.
\end{equation}
We start by verifying that the operator $\tilde{H}_{\varepsilon,h_k}$ satisfies Assumption~\ref{Assumption:local_potential_1}. From Lemma~\ref{LE:framing_operators} it follows that the operator $\tilde{H}_{\varepsilon,h_k}$ is lower semibounded and selfadjoint. By our choice of $\varphi_1$ we have that $\tilde{H}_{\varepsilon,h_k}$ will satisfies part two of Assumption~\ref{Assumption:local_potential_1} with the set $B(0,8)$ and the potential
\begin{equation}\label{EQ:proof_main_8}
	\widetilde{\varphi_1V}_\varepsilon(x) = \varphi_1(l_kx+x_k) f_k^{-2}V_\varepsilon(l_kx+x_k),
\end{equation}
where we by \eqref{EQ:proof_main_8} have that $\widetilde{\varphi_1V}_\varepsilon(x)\in C_0^\infty(\R^d)$.  What remains to verify is that we have obtained a non-critical condition \eqref{THM:model_prob_global_Non_crit}. Using \eqref{EQ:proof_main_4} we have for for $x$ in $B(0,8)$ that
\begin{equation*}
	\begin{aligned}
	\abs{\widetilde{\varphi_1V}_\varepsilon(x)} + h_k^{\frac{2}{3}} &= f_k^{-2} \abs{\varphi_1V_\varepsilon(l_kx+x_k)} + (\tfrac{\hbar}{f_k l_k})^{\frac{2}{3}}
	=l_k^{-1}( \abs{\varphi_1V_\varepsilon(l_kx+x_k)} +\hbar^{\frac23})
	\\
	&\geq  l_k^{-1} A l(l_k x+x_k) \geq (1-8\rho) A.
	\end{aligned}
\end{equation*}
Hence we have obtained the  non-critical condition on $B(0,8)$. So all assumptions of Theorem~\ref{THM:Loc_mod_prob} is fulfilled. But before applying it we will verify that the numbers the constant from Theorem~\ref{THM:Loc_mod_prob} depends on are independent of $k$ and $\hbar$. Firstly we have for the norm estimate for the potential that
\begin{equation*}
	\norm{\widetilde{\varphi_1V}_\varepsilon}_{L^\infty(B(0,8))} = \sup_{x\in B(0,8)} \big|  \varphi_1(l_kx+x_k) f_k^{-2}V_\varepsilon(l_kx+x_k)\big| \leq (1+8\rho)A,
\end{equation*}
where we have used \eqref{EQ:proof_main_4} and \eqref{EQ:Rough_weyl_asymptotics_3.5}. When considering the derivatives we have for $\alpha\in\N_0^d$ with $|\alpha|\geq1$ that
\begin{equation}
	\begin{aligned}
    	\MoveEqLeft \varepsilon_k^{-\min(0,\tau-|\alpha|)}\norm{\partial^\alpha\widetilde{\varphi_1V}_\varepsilon}_{L^\infty(\R^d)}
	\\
	&\leq \varepsilon_k^{-\min(0,\tau-|\alpha|)}  f_k^{-2} l_k^{|\alpha|}  \varepsilon^{\min(0,\tau-|\alpha|)} \sup_{x\in\R^d} \sum_{\beta\leq \alpha } \binom{\alpha}{\beta} \big| (\partial^{\alpha-\beta}\varphi_1)(\partial^\beta V_\varepsilon)(l_kx+x_k) \big|  
     	\\
     	& \leq   C_\alpha ,
     	\end{aligned}
\end{equation}
where $C_\alpha$ is independent of $k$ and $\hbar$. We have in the estimate used the definition of $\varepsilon_k$, $f_k$, \eqref{EQ:proof_main_7} and Proposition~\ref{PRO:smoothning_of_func}. Hence we have that all these estimates are independent of $\hbar$ and $k$. The last numbers we check are the numbers $\norm{\partial_x^\alpha \widetilde{\varphi_k\varphi}}_{L^\infty(\R^d)}$ for all $\alpha\in\N_0^d$, where $\widetilde{\varphi_k\varphi}=(T_{x_k} U_{l_k})\varphi_k\varphi(T_{x_k} U_{l_k})^{*}$. Here we have by construction of $\varphi_k$ \eqref{EQ:proof_main_5} for all $\alpha\in\N_0^d$ 
\begin{equation*}
	\begin{aligned}
   	\norm{\partial_x^\alpha \widetilde{\varphi_k\varphi}}_{L^\infty(\R^d)}
   	&=\sup_{x\in\R^d} \abs{l_k^{\abs{\alpha}} \sum_{\beta\leq \alpha} {\binom{\alpha}{\beta}} (\partial_x^{\beta}\varphi_k)(l_k x+x_k) (\partial_x^{\alpha - \beta}\varphi)(l_kx+x_k) }
    	\\
    	&\leq C_\alpha \sup_{x\in\R^d} \sum_{\beta\leq \alpha} {\binom{\alpha}{\beta}} l_k^{\abs{\alpha-\beta} }\abs{(\partial_x^{\alpha - \beta}  \varphi)(l_kx+x_k) } \leq \widetilde{C}_\alpha.
      \end{aligned}
  \end{equation*}
With this we have established that all numbers the constant from Theorem~\ref{THM:Loc_mod_prob} are independent of $\hbar$ and $k$.
From applying Theorem~\ref{THM:Loc_mod_prob} we get that
\begin{equation}\label{EQ:proof_main_9}
	\begin{aligned}
	\MoveEqLeft \big| \Tr\big[\varphi  g_\gamma (H_{\varepsilon,\hbar})  \big] - \frac{1}{(2\pi\hbar)^d} \int_{\R^{2d}}g_\gamma( p^2+V(x))\varphi(x) \,dx dp \big|
	\\
	\leq {}& \sum_{k\in\mathcal{I}}\big|  \Tr\big[\varphi_k\varphi  g_\gamma (H_{\varepsilon,\hbar} )  \big] - \frac{1}{(2\pi\hbar)^d} \int_{\R^{2d}}g_\gamma( p^2+V(x))\varphi_k\varphi(x) \,dx dp \big|
	\\
	\leq  {} & \sum_{k\in\mathcal{I}} f_k^{2\gamma}  \big|\Tr\big[ g_\gamma\tilde{H}_{\varepsilon,h_k} ) \widetilde{\varphi_k\varphi} \big]
	 - \frac{1}{(2\pi h_k)^d} \int_{\R^{2d}} g_\gamma( p^2+\tilde{V}(x))\widetilde{\varphi_k\varphi}(x) \,dx dp \big|
	 \\
	  \leq  {} & \sum_{k\in\mathcal{I}} f_k^{2\gamma}  \big|\Tr\big[ g_\gamma (\tilde{H}_{\varepsilon,h_k} ) \widetilde{\varphi_k\varphi} \big]
	 - \frac{1}{(2\pi h_k)^d} \int_{\R^{2d}} g_\gamma( p^2+\widetilde{\varphi_1 V_\varepsilon}(x))\widetilde{\varphi_k\varphi}(x) \,dx dp \big|
	 \\
	 &+ \sum_{k\in\mathcal{I}} \frac{ f_k^{2\gamma}}{(2\pi h_k)^d}   \Big|\int_{\R^{2d}} \big[ g_\gamma( p^2+\widetilde{\varphi_1 V_\varepsilon}(x))\widetilde{\varphi_k\varphi}(x) -g_\gamma( p^2+\tilde{V}(x)) \big]\widetilde{\varphi_k\varphi}(x) \,dx dp \Big|
	\\
	\leq {} & C \sum_{k\in\mathcal{I}} \frac{f_k^{2\gamma} }{h_k^{d}} \Big[ h_k^{1+\gamma} +    \Big|\int_{\R^{2d}} \big[ g_\gamma( p^2+\widetilde{\varphi_1 V_\varepsilon}(x))\widetilde{\varphi_k\varphi}(x) -g_\gamma( p^2+\tilde{V}(x)) \big]\widetilde{\varphi_k\varphi}(x) \,dx dp \Big|\Big].
	\end{aligned}
\end{equation}
To estimate the remaining integrals we will use Lemma~\ref{LE:comparison_phase_space_int}. Combining this Lemma with \eqref{EQ:proof_main_0} we obtain that
\begin{equation}\label{EQ:proof_main_10}
	\begin{aligned}
	\MoveEqLeft   \Big|\int_{\R^{2d}} \big[ g_\gamma( p^2+\widetilde{\varphi_1 V_\varepsilon}(x))\widetilde{\varphi_k\varphi}(x) -g_\gamma( p^2+\tilde{V}(x)) \big]\widetilde{\varphi_k\varphi}(x) \,dx dp \Big| \leq C \hbar^{1+\gamma} \leq C h_k^{1+\gamma} . 
	\end{aligned}
\end{equation}
Hence we obtain from combining \eqref{EQ:proof_main_9} and \eqref{EQ:proof_main_10} that
\begin{equation}\label{EQ:proof_main_11}
	\begin{aligned}
	\big| \Tr\big[\varphi  g_\gamma (H_{\varepsilon,\hbar})  \big] - \frac{1}{(2\pi\hbar)^d} \int_{\R^{2d}}g_\gamma( p^2+V(x))\varphi(x) \,dx dp \big|
	\leq C \sum_{k\in\mathcal{I}} f_k^{2\gamma}  h_k^{1+\gamma-d} .
	\end{aligned}
\end{equation}
By just considering the sum over $k$ on the righthand side of \eqref{EQ:proof_main_11} and by using \eqref{EQ:Rough_weyl_asymptotics_3.5} we have that
\begin{equation}\label{EQ:proof_main_12}
	\begin{aligned}
	 \sum_{k\in\mathcal{I}} C h_k^{1+\gamma-d}f_k^{2\gamma} &=  \sum_{k\in\mathcal{I}} \tilde{C} \hbar^{1+\gamma-d} \int_{B(x_k,l_k)} l_k^{-d} f_k^{2\gamma}(l_kf_k)^{d-1-\gamma} \,dx
	\\
	& =  \sum_{k\in\mathcal{I}} \tilde{C} \hbar^{1+\gamma-d} \int_{B(x_k,l_k)} l_k^{\gamma-d} l_k^{\frac{3d-3-3\gamma}{2}} \,dx 
	\\
	&\leq \sum_{k\in\mathcal{I}} \hat{C} \hbar^{1+\gamma-d} \int_{B(x_k,l_k)} l(x)^{\frac{d -3 -\gamma}{2}}\,dx
	\leq C   \hbar^{1+\gamma-d}, 
	\end{aligned}
\end{equation}%
where we in the last inequality have used that $\supp(\varphi)\subset\Omega_{2\tilde{\nu},V_\varepsilon}$ and that $\Omega_{2\tilde{\nu},V_\varepsilon}$ is assumed to be compact. This ensures that the constant obtained in the last inequality is finite. From combing the estimates and identities in \eqref{EQ:proof_main_1}, \eqref{EQ:proof_main_2}, \eqref{EQ:proof_main_11} and \eqref{EQ:proof_main_12} we obtain that
\begin{equation*}
	  \Big|\Tr\big[\boldsymbol{1}_{ (-\infty,0]}(H_{\hbar,\varepsilon})\big] - \frac{1}{(2\pi\hbar)^d} \int_{\R^{2d}}\boldsymbol{1}_{(-\infty,0]}(p^2+V_\varepsilon(x))\,dx dp \Big| \leq C \hbar^{1-d}
\end{equation*}
for all $\hbar\in(0,\hbar_0]$. This concludes the proof.
\end{proof} 
\begin{proof}[Proof of Theorem~\ref{Thm:Main_2} and Theorem~\ref{Thm:Main_3}]
The proofs are almost analogous to the proof just given for Theorem~\ref{Thm:Main}. The difference is that $\delta$ is here always chosen to be $\frac{1}{3}$ when choosing the scaling of the framing operators $H_{\hbar,\varepsilon}^{\pm}$ with $\varepsilon=\hbar^{1-\delta}$. After this choice the remainder of the proof is identical. This concludes the proof.
\end{proof}

\begin{proof}[Proof of Theorem~\ref{Thm:Main_4} and Theorem~\ref{Thm:Main_5}]
The proofs are again almost analogous to the proof just given for Theorem~\ref{Thm:Main}. We have the same difference as before. That is we always choose $\delta=\frac{1}{3}$ when choosing the scaling of the framing operators $H_{\hbar,\varepsilon}^{\pm}$ with $\varepsilon=\hbar^{1-\delta}$. After this choice the remainder of the proof is identical until \eqref{EQ:proof_main_12}. For the cases considered here we have that $d-3-\gamma<0$. Hence we get a negative power of $l(x)$ and will have to use the lower bound $l(x)\geq C \hbar^{\frac{2}{3}}$ and not that $l(x)$ is bounded from above. Using this and by calculating the power we obtain of the semiclassical parameter one obtains the errors stated in the theorems. This concludes the proof.
\end{proof}

\bibliographystyle{plain} \bibliography{Bib_paperB.bib}
\end{document}